\documentclass{amsart}

\usepackage{ifthen}
\usepackage{graphicx}
\usepackage{amssymb}
\usepackage{enumerate}
\usepackage{url}
\newtheorem{anyprop}{Anyprop}[section]

\newtheorem{theorem}[anyprop]{Theorem}

\theoremstyle{definition}

\newtheorem{remark}[anyprop]{Remark}

\theoremstyle{remark}

\numberwithin{equation}{section}

\usepackage{amscd}
\usepackage{amssymb}
\input xy
\xyoption{all}

\begin{document}
\title[NULLSTELLENSATZ FOR ARBITRARY FIELDS]
{A GENERAL VERSION OF THE NULLSTELLENSATZ FOR ARBITRARY FIELDS}

%\author[Edisson Gallego] {Edisson Gallego}

\author[Juan D. Velez]{Juan D. V\'elez}
\author[Danny de Jes\'us G\'omez-Ram\'irez]{Danny de Jes\'us G\'omez-Ram\'irez}
\author[Edisson Gallego]{Edisson Gallego}
\address{Universidad Nacional de Colombia, Escuela de Matem\'aticas,
Calle 59A No 63 - 20, N\'ucleo El Volador, Medell\'in, Colombia.}
\address{Vienna University of Technology, Institute of Discrete Mathematics and Geometry,
wiedner Hauptstaße 8-10, 1040, Vienna, Austria.}
\address{University of Antioquia, Calle 67 \# 53-108, Medell\'in, Colombia}
\email{jdvelez@unal.edu.co}
\email{daj.gomezramirez@gmail.com}
\email{egalleg@gmail.com}

%\address{Universidad Nacional de Colombia, Escuela de Matem\'aticas,
%Calle 59A No 63 - 20, N\'ucleo El Volador, Medell\'in, Colombia.}
%\email{egalleg@unal.edu.co}%\thanks{to my parents}

%\subjclass{14A10 (primary), 13A15 (secondary)}

\begin{abstract}
We prove a general version of Bezout's form of the Nullstellensatz for arbitrary fields. 
The corresponding sufficient and necessary condition only involves the local existence of multi-valued 
roots for each of the polynomials belonging to the ideal in consideration. Finally, this version implies
the standard Nullstellensatz when the coefficient field is algebraically closed.

\end{abstract}

\maketitle


\smallskip


\section*{Introduction}

A fundamental result in algebraic geometry is the well-known Hilbert's Nullstellensatz, which describes the interrelations between an ideal $I$ in a polynomial ring over an algebraically closed field $k$ and the corresponding ideal of all polynomials vanishing on the zeroes of $I$, i.e. $I(V(I))=Rad(I)$ \cite[Ch. 1]{hartshornealgebraic}. Besides, the Nullstellensatz has also a `weak' form (or Bezout version) which states that under the former hypothesis an ideal $I$ in $k[x_1,\cdots, x_n]$ contains $1$ if and only if there is no common zero for all the polynomials of $I$ in $k^n$. In addition, the standard version of the Nullstellensatz can be easily deduced from the weak form using the Rabinowitsch's trick \cite{rabinowitsch}.

In the literature one can find several kinds of generalizations of both forms of this seminal result, as for example, a noncommutative version due to S. A. Amitsur \cite{amitsur}; the work of W. D. Brownawell describing a corresponding ``pure power product version'', which relates in a sophisticated way the exponents emerging from the `radical' condition stated in `homogeneous' forms of the Nullstellensatz \cite{brownawell}. In addition, the main result of T. Krick et al. in \cite{krick} offers sharp bounds for the degree and the height of the polynomials involved in the arithmetic (weak) form of the Nullstellensatz, and the work of L. Ein and R. Lazarsfeld proves more sophisticated geometric versions of it involving ideal sheaves, among others \cite{ein}. Finally, from the Artin-Tate lemma it can be derived a more general form of the Nullstellensatz for arbitrary fields, i.e., the quotient of a polynomial ring in finitely many variables over a field $L$ by a maximal ideal $m$, is a finite field extension of $L$. This is an elementary consequence of the Artin-Tate lemma and the Steinitz theorem (see for example \cite{wnull}). 

Now, let us assume that $k$ is an arbitrary ideal. Then, what will be the natural condition for $I$ characterizing the fact that $V(I)$ is non-empty?

Surprisingly, none of the results above offers an answer to this elementary question, whose answer can be considered genuinely as a formal generalization of the (weak) Nullstellensatz for arbitrary fields. 

So, in this short communication we prove in a completely elementary way that the global non-emptyness of the zero-locus of an ideal of polynomials over any field is equivalent to the local non-emptyness of the zero-locus of any of its elements, which would seem to be, in general, a strictly weaker condition. This result would offer a natural extension of a general condition characterizing the global non-emptyness of the zero-locus of ideals independently of the coefficient field $k$.

\section{Main Result}

\begin{theorem}[Bezout's Form of the Nullstellensatz for arbitrary fields]
Let $k$ be a field (not necessarily closed) and $I\subset k[X_{1},\ldots
,X_{n}]$ an ideal. Then for each polynomial $f(X_{1},...,X_{n})\in I$,
there exists $a=(a_{1},...,a_{n})\in k^{n}$ such that $f(a)=0$, if and only if the
algebraic set determined by $I,$ $V(I)$, is non-empty.
\end{theorem}

\begin{proof}

Case I) $k$ is algebraically closed: this corresponds to the classic version
of the (weak) Nullstellensatz.

Case II) $k$ is not algebraically closed. In this case we state that given
polynomials $f_{1}$ and $f_{2}\in I$, there is another polynomial $%
p(f_{1},f_{2})\in I$ such that $p(f_{1},f_{2})(a)=0\Leftrightarrow
f_{1}(a)=f_{2}(a)=0.$

To see this, let $l(T)=T^{n}+a_{1}T^{n-1}+\cdots +a_{n}\in k[T]$ be any
monic non constant polynomial without roots in $k$. We define

\[
P(f_{1},f_{2})(X) =f_{2}^{n}(X)\left( \left( \frac{f_{1}(X)}{f_{2}(X)}%
\right) ^{n}+a_{1}\left( \frac{f_{1}(X)}{f_{2}(X)}\right) ^{n-1}+\cdots
+a_{n}\right) \]

\[=f_{1}^{n}+a_{1}f_{2}f_{1}^{n-1}+\cdots +a_{k}f_{1}^{n-k}f_{2}^{k}+\cdots
+a_{n}f_{2}^{n}\]

It is clear that $P(f_{1},f_{2})\in I$. Let us prove that $%
P(f_{1},f_{2})(a)=0\Leftrightarrow f_{1}(a)=f_{2}(a)=0.$ Suppose $%
P(f_{1},f_{2})(a)=0$. Then $f_{2}(a)\not=0$, since otherwise $%
f_{1}(a)/f_{2}(a)$ would be a root of $l(T)$. Hence, $f_{2}(a)=0,$ and since 
$P(f_{1},f_{2})(a)=0$, it follows that $f_{1}^{n}(a)=f_{1}(a)=0.$

The reciprocal is clear, since $f_{1}(a)=f_{2}(a)=0$ implies $%
p(f_{1},f_{2})(a)=0.$

Finally, let $f_{1},...,f_{n}$ be arbitrary generators of $I.$ We
inductively define $p_{1}=p(f_{1},f_{2}),...,p_{r}=p(f_{r},p_{r-1})$.
Clearly $p_{r}(a)=0$ if and only if $f_{1}(a)=\cdots =f_{n}(a)=0$. Since $%
p_{n}\in I$, the hypothesis guarantees the existence of $a\in k^{n}$ such
that $p_{n}(a)=0$. Thus, $f_{1}(a)=\cdots =f_{n}(a)=0,$ whereas it follows clearly that 
$g(a)=0,$ for all $g\in I$.
\end{proof}

\begin{remark}
If $k$ is an algebraically closed field, the hypothesis of the theorem are
satisfied under the standard assumption that $I\neq R$ and consequently it
generalizes the classic (weak) Nullstellensatz. Effectively, if $f$ is a non-constant 
polynomial in $I$, let us see that the zero-locus of $f$ should be non-empty. 
So, after a standard change of coordinates it is possible to write $f$ as a
monic polynomial in one of the variables, let's say $X_{n}.$ 
Hence we may assume that $f(X_{1},...,X_{n})$ can be written
in the form%
\[
X_{n}^{r}+a_{1}(X_{1},...,X_{n-1})X_{n}^{r-1}+...+a_{r}(X_{1},...,X_{n-1}).
\]

Substituing $X_{1}=X_{2}=\cdots =X_{n-1}=0$ we obtain the polynomial $%
X_{n}^{r}+a_{1}(0)X_{n}^{r-1}+\cdots +a_{r}(0).$ Since we are assuming $k$
is algebraically closed, it has a zero $c\in k.$ Thus, in these new
coordinates $f$ has a cero $(0,...0,c)\in k^{n}.$
\end{remark}

\section*{Acknowledgements}
The authors would like to thank to the Universidad Nacional Colombia in Medell\'in and to the University of Osnabr\"uck for all the support. Danny Arlen de Jes\'us G\'omez-Ram\'irez was supported by the Vienna Science and Technology Fund (WWTF) as part of the Vienna Research Group 12-004. Finally, Danny A. J. G\'omez-Ram\'irez would like to thank Felipe  and Carlos Arroyave, Lucila Mendez and Y. G\'omez for their happiness and support.
\bibliographystyle{amsplain}

\end{document}